\documentclass[11pt]{amsart}

\usepackage[left=2.5cm, right=2.5cm, top=2.5cm, bottom=2.5cm, bindingoffset=0cm]{geometry}

\usepackage{amsthm, amssymb, amsmath, amsfonts, amsaddr}

\usepackage{hyperref}

\usepackage{graphicx}
\usepackage{pgf, tikz, pgfplots, tikz-cd}
\pgfplotsset{compat=1.15}
\usetikzlibrary{arrows}

\usepackage{mathrsfs}
\usepackage{dynkin-diagrams, ytableau}
\usepackage{multirow}

\allowdisplaybreaks

\newcommand{\rge}{\rangle}
\newcommand{\lge}{\langle}

\newcommand{\mc}{\mathcal}

\newcommand{\supp}{\operatorname{supp}}
\newcommand{\GL}{\operatorname{GL}}

\renewcommand{\le}{\leqslant}
\renewcommand{\ge}{\geqslant}

\newcommand{\fullstopbelow}{\makebox[0pt][l]{\,.}} 
\newcommand{\commabelow}{\makebox[0pt][l]{\,,}}    

\newcommand{\C}{\mathbb{C}}

\newcommand{\Z}{\mathbb{Z}}

\newcommand{\F}{\mathbb{F}}

\newcommand{\g}{\mathfrak{g}}

\newcommand{\git}{/\!\!/}

\numberwithin{equation}{section}
\theoremstyle{plain}
\newtheorem{theorem}{Theorem}[section]
\newtheorem{definition}[theorem]{Definition}
\newtheorem{proposition}[theorem]{Proposition}
\newtheorem{corollary}[theorem]{Corollary}
\newtheorem{lemma}[theorem]{Lemma}
\newtheorem{remark}[theorem]{Remark}
\newtheorem{example}[theorem]{Example}

\newtheorem{conjecture}[theorem]{Conjecture}

\title[Stabilization of Kac polynomials along root strings]{Stabilization of Kac polynomials along root strings}
\author[Vladyslav Zveryk]{Vladyslav Zveryk}
\address{Yale University}
\email{zverik.vladislav@gmail.com}

\begin{document}

\begin{abstract}
We study the stabilization behavior of cohomology groups associated with moduli spaces of quiver representations for a fixed quiver $Q$. Under mild conditions on a dimension vector $\delta$, we show that the dimensions of these cohomology groups stabilize when a sufficiently large multiple of $\delta$ is added. We derive explicit formulas for the stabilized dimensions and, in particular, obtain stabilization of the coefficients of Kac polynomials.
\end{abstract}

\maketitle

\section{Introduction}
Kac polynomials, introduced by Kac in \cite{Kac}, encode the enumeration of absolutely indecomposable representations of quivers over finite fields and establish connections between quiver representations and the root system of the associated Kac-Moody Lie algebras. The degree and coefficients of these polynomials provide information about the representation theory of quivers, particularly in relation to the root multiplicities of the Lie algebra.

In this paper, we study the stabilization behavior of Kac polynomials when a sufficiently large multiple of a fixed dimension vector $\delta$ is added. We consider quivers $Q$ without edge loops and dimension vectors $d$ and $\delta$ under conditions ensuring that $d+n\delta$ is an imaginary root of $Q$ for large $n$. 
With a certain assumption on $\delta$ (see (\ref{eq:condition})), our main result is the following (Theorem \ref{t:kacstab}):
\begin{theorem}
    Write the Kac polynomial corresponding to a dimension vector $d+n\delta$ as
    $$
    A_{d+n\delta}(q)=\sum_{i=0}^{\deg A_{d+n\delta}}a_i^{(d+n\delta)}q^{\deg A_{d+n\delta}-i}.
    $$
    Then for every $i$ the coefficients $a_i^{(d+n\delta)}$ stabilize for sufficiently large $n$. Moreover, the generating function of the stabilized coefficients is given by
    \[
    \frac{(1-q) p^{|\operatorname{supp} \delta|}(q)}{\prod_{i \notin \operatorname{supp} \delta} \prod_{k=1}^{d_i} (1 - q^k)},
    \]
    where $p^l(q)$ denotes the generating function of the partition number into $l$ colors.
\end{theorem}

We also prove similar results for
\begin{itemize}
    \item Moduli spaces parametrising representations of $Q$ (Theorem \ref{t:cohomostab}).
    \item Nakajima varieties associated with $Q$ (Theorem \ref{t:nakajimastab}).
\end{itemize}
In addition, we explain the formulas for the stabilized coefficients in terms of Kirwan surjectivity.

In Section 2, we review the necessary background on Kac polynomials, quiver varieties, and cohomological interpretations of quiver moduli. The definitions of representation varieties and their semistability are presented, along with a summary of key results from geometric invariant theory and Harder-Narasimhan stratification (\cite{Rei08}, \cite{Kir}). The section concludes with a discussion of Kirwan surjectivity and its implications for quiver varieties \cite{MGN18}.

In Section 3, we establish sufficient conditions for stabilization of cohomology dimensions under the addition of large multiples of $\delta$. Specifically, we prove that for sufficiently large $n$, the cohomology dimensions of the moduli space of semistable quiver representations stabilize, and provide a formula for the stabilized dimensions. We also provide an explanation of our result in terms of Kirwan surjectivity.

In Section 4, we focus on the stabilization of Kac polynomials. We derive an explicit formula for the stabilized coefficients of Kac polynomials. We also provide a cohomological interpretation of the stabilized coefficients using quiver varieties.

In Section 5, we extend our results to cases where the conditions on $\delta$ are relaxed. This generalization allows the stabilization results to be applied to a broader class of dimension vectors, including isotropic roots.

In Section 6, we prove similar results for Nakajima quiver varieties, utilizing the Crawley-Boevey trick (\cite{CB01}) to express Nakajima quiver varieties as moduli spaces of quiver representations with an additional vertex. We formulate several conjectures supported by computer computations and illustrate some of them using Hilbert schemes of points.

\section{Kac Polynomials and Quiver Varieties}

Let $Q = (Q_0, Q_1)$ be a quiver, where $Q_0$ is the set of vertices and $Q_1$ is the set of edges. We denote by $\lge \cdot, \cdot \rge$ its \textbf{Euler form}, which is defined as
\[
\lge d, v \rge := \sum_{i \in Q_0} d_i v_i - \sum_{i \to j} d_i v_j.
\]
Similarly, the {\bfseries Cartan form} of $Q$ is defined as
$$
(d,v):=\lge d,v\rge+\lge v,d\rge.
$$

We denote by $e_i$ the dimension vector having one at vertex $i\in Q_0$ and $0$ at other vertices. 
\subsection{Representation Varieties and Stability}
We recall the basic invariant theory of quiver representations. We recommend \cite{Rei08} and \cite{Kir} as references for this section.

For a dimension vector $d$, let $\operatorname{Rep}_d(Q)$ be the affine variety parametrizing all representations of $Q$ of dimension $d$. The coordinate ring of $\operatorname{Rep}_d(Q)$ is naturally a product of matrix entries corresponding to the arrows in $Q$. The group
\[
\GL_d = \prod_{i \in Q_0} \operatorname{GL}_{(d_i)}
\]
acts on $\operatorname{Rep}_d(Q)$ by base change in each vertex space.

Fix a character $\chi$ of $G$, which is known to have the form 
\[
(g_i)_{i\in Q_0} \mapsto \prod_{i\in Q_0} (\det g_i)^{\chi_i}
\]
for some $\chi_i\in\Z$. We identify $\chi$ with $(\chi_i)_{i\in Q_0}$.

We define the \textbf{slope} of $d$ as
\[
s(d) := \frac{\chi(d)}{|d|},
\]
where
\[
|d| = \sum_i d_i, \quad \chi(d) = \sum_i \chi_i d_i.
\]

\begin{definition}\label{def:stable}
    A representation $V$ in $\mathrm{Rep}_d(Q)$ is called {\bfseries ($\chi$-)semistable} if for any subrepresentation $W\subset V$ which is neither $0$ nor $V$, 
    $$
    s(\dim W)\le s(\dim V).
    $$
    If the inequality above is strict for all such $W$, the representation $V$ is called {\bfseries ($\chi$-)stable}.
\end{definition}

We can form the following projective map of GIT quotients:
\[
R_d^\chi(Q) := \mathrm{Rep}_d(Q) \git_\chi \GL_d\to \mathrm{Rep}_d(Q) \git \GL_d=:R_d^0(Q).
\]
We summarize the key properties of these varieties:
\begin{theorem}[\cite{Rei08}, 3.5]\
    \begin{itemize}
        \item [(i)] The subsets $\operatorname{Rep}_d(Q)^{ss}$ and $\operatorname{Rep}_d(Q)^{st}$ of semistable and stable representations, respectively, are open in $\operatorname{Rep}_d(Q)$.
        \item [(ii)] We have a canonical surjective map $\operatorname{Rep}_d(Q)^{ss}\to R_d^\chi(Q)$ that factors through $\operatorname{Rep}_d(Q)^{ss}/\GL_d$. Two semistable $\GL_d$-orbits give the same point in $R_d^\chi(Q)$ if and only if their closures intersect in $\operatorname{Rep}_d(Q)^{ss}$.
        \item [(iii)] If $\operatorname{Rep}_d(Q)^{st}$ is nonempty, then $\GL_d/\mathrm{diag}(\mathbb G_m)$ acts on it freely and $\operatorname{Rep}_d(Q)^{st}/\GL_d$ is a smooth open subvariety of $R_d^\chi(Q)$ of dimension $1-\lge d,d\rge$.
    \end{itemize}
\end{theorem}
\subsection{Representation varieties for the double quiver}
Let $\bar Q$ be the double quiver for $Q$. We can identify 
$$
\mathrm{Rep}_d(\bar Q)\simeq\bigoplus_{Q_1\ni\alpha:i\to j} T^*\mathrm{Rep}_d(Q),
$$
which has a moment map 
$$
\mu_d: \mathrm{Rep}_d(\bar Q)\to \mathfrak{gl}_d:=\prod_{i\in Q_0}\mathfrak{gl}_{d_i}
$$
such that its $k$-th component sends $(B_\alpha, B_{\alpha^*})_{\alpha\in Q_1}$ to
$$
\sum_{Q_1\ni\alpha:j\to k}B_\alpha B_{\alpha^*}-\sum_{Q_1\ni\alpha:k\to j}B_{\alpha^*} B_\alpha.
$$

Hamiltonian reduction with respect to this moment map gives the following projective map of varieties:
$$
\mc M_d^\chi(Q) := \mu^{-1}_d(0) \git_\chi \GL_d \to \mu^{-1}_d(0) \git \GL_d=:\mc M_d^0(Q).
$$

If the $\chi$-semistable locus for $\mu^{-1}_d(0)$ is nonempty, we have a commutative diagram
\begin{center}
    \begin{tikzcd}
R_d^\chi(\bar Q) \arrow[r, two heads]                & R_d^0(\bar Q)                \\
\mc M_d^\chi(Q) \arrow[u, hook] \arrow[r, two heads] & \mc M_d^0(Q) \arrow[u, hook]\commabelow
\end{tikzcd}
\end{center}
where the horizontal maps are projective surjections and the vertical maps are closed immersions.

\begin{theorem}[\cite{Kir}, Theorem 10.12] 
If the stable locus $\mu_d^{-1}(0)^{st}$ is nonempty, then $\mu_d^{-1}(0)^{st}/\GL_d$ is a smooth open subvariety of $\mc M_d^\chi(Q)$ of dimension $2-2\lge d,d\rge$.
\end{theorem}

\subsection{Nakajima quiver varieties}
We recall the construction of the Nakajima quiver varieties introduced in \cite{Nak98}.

Let $Q_F$ denote the quiver obtained from $Q$ by creating a vertex $1_i$ for each $i\in Q_0$ and drawing one arrow from $1_i$ to $i$ for every $i$. We call $Q_F$ the {\bfseries framed quiver for $Q$}. We denote its double by $\bar Q_F$.
\begin{example}\label{e:hyperbolicquiver}
    For the quiver
    \[\begin{tikzcd}
	\bullet & \bullet & \bullet
	\arrow[shift right=1, from=1-1, to=1-2]
	\arrow[shift left=1, from=1-1, to=1-2]
	\arrow[from=1-2, to=1-3],
\end{tikzcd}\]
the associated framed quiver is
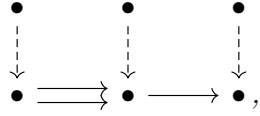
\[\begin{tikzcd}
	\bullet & \bullet & \bullet \\
	\bullet & \bullet & \bullet
	\arrow[dashed, from=1-1, to=2-1]
	\arrow[shift right, from=2-1, to=2-2]
	\arrow[shift left, from=2-1, to=2-2]
	\arrow[dashed, from=1-2, to=2-2]
	\arrow[from=2-2, to=2-3]
	\arrow[dashed, from=1-3, to=2-3]\commabelow
\end{tikzcd}\]
where we drew dashes to indicate the added arrows.
\end{example}

We choose a dimension vector for $Q_F$, which we denote by $(d,w)$, where $d$ is the dimension subvector of the subquiver $Q$ and $w$ are dimensions corresponding to the added vertices $1_i$. We call $w$ a {\bfseries framing} and write $w_i$ for the component corresponding to the vertex $1_i$.

We consider the $\GL_d$ action on $\operatorname{Rep}_{d,w}(Q_F)$ and  $\operatorname{Rep}_{d,w}(\bar Q_F)$ (recall that $\bar Q_F$ is the double quiver for $Q_F$) by base change only in the vertices of $Q$. There is a moment map
$$
\mu_{d,w}:\operatorname{Rep}_{d,w}(\bar Q_F)\to \mathfrak{gl}_d
$$
whose $k$-th component sends $(B_\alpha, B_{\alpha^*})_{\alpha\in (Q_F)_1}$ to
$$
\sum_{(Q_F)_1\ni\alpha:j\to k}B_\alpha B_{\alpha^*}-\sum_{(Q_F)_1\ni\alpha:k\to j}B_{\alpha^*} B_\alpha.
$$
The corresponding {\bfseries Nakajima quiver variety} is defined as
$$
\mc M^\chi(d,w):=\mu_{d,w}^{-1}(0)\git_\chi\GL_d.
$$

It turns out that there is a way to realize Nakajima quiver varieties as particular examples of the varieties $\mc M_d^\chi(Q)$ described above, called the Crawley-Boevey trick \cite[pp. 261-262]{CB01}. For a quiver $Q$ with dimension vector $d$ and framing $w\ne 0$, we define a new quiver $Q_w$ in the following way:
\begin{itemize}
    \item The set of vertices of $Q_w$ equals $Q_0\cup\{\infty\}$, with a vertex $\infty$ added to $Q$
    \item The set of edges of $Q_w$ equals $Q_1$ plus $w_i$ edges from $\infty$ to the vertex $i$ for any vertex $i\in Q_0$.
\end{itemize}
We call this quiver the {\bfseries Crawley-Boevey quiver} for $Q$ and the framing $w$.
\begin{example}
    Let $Q$ be the quiver from Example \ref{e:hyperbolicquiver} with framing $w=(2,0,1)$. Then the quiver $Q_w$ is

   
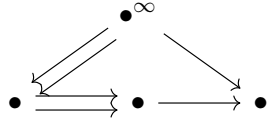
\[\begin{tikzcd}
	& {\bullet^\infty} \\
	\bullet & \bullet & \bullet
	\arrow[shift left, from=1-2, to=2-1]
	\arrow[shift right, from=1-2, to=2-1]
	\arrow[from=1-2, to=2-3]
	\arrow[shift right, from=2-1, to=2-2]
	\arrow[shift left, from=2-1, to=2-2]
	\arrow[from=2-2, to=2-3]
\end{tikzcd}\]
\end{example}

Denote by $(d,1)$ the dimension vector of $Q_w$, where we assign $d$ to the subquiver $Q$ and $1$ to the additional vertex $\infty$.
\begin{proposition}[pp. 261-262, \cite{CB01}]
    We have an isomorphism
    $$
    \mc M^\chi(d,w)\simeq\mc M_{(d,1)}^\chi(Q_w).
    $$
\end{proposition}

\subsection{Harder-Narasimhan Stratification}
In this section, we fix a stability parameter $\chi$. We follow \cite{Rei03}. The non-semistability of a representation $V$ of a quiver $Q$ (with respect to a chosen character $\chi$) is measured by the so-called Harder-Narasimhan filtration.

\begin{definition}
    A \textbf{Harder-Narasimhan (HN) filtration} of $V$ is a filtration
    \[
    0 = V_0 \subset V_1 \subset \cdots \subset V_n = V
    \]
    such that the quotients $V_i / V_{i+1}$ are semistable and
    \[
    s(\dim V_1 / V_0) > s(\dim V_2 / V_1) > \cdots > s(\dim V_n / V_{n-1}).
    \]
    
    Let $d^1, \dots, d^k$ be dimension vectors, and let $d := \sum d^i$. The corresponding \textbf{Harder-Narasimhan (HN) stratum} $\operatorname{Rep}_{(d^i)}^{HN}(Q)$ is defined as the set of all representations $V \in \operatorname{Rep}_d(Q)$ whose HN filtration has dimensions $d^1, d^1 + d^2, \dots, d$.

    Denote by $\operatorname{Rep}_{(d^i)}(Q)$ the set of all representations $V \in \operatorname{Rep}_d(Q)$ possessing a filtration with dimensions $d^1, d^1 + d^2, \dots, d$.
\end{definition}

The main properties of these objects are summarized in the following result:

\begin{theorem}\label{t:HNfiltandstrat}\
    \begin{itemize}
        \item[(i)] \cite[Proposition 2.5]{Rei03} The HN filtration of any $V \in \operatorname{Rep}_{d}(Q)$ is unique. It has only one component (namely, $V$ itself) if and only if $V$ is semistable.
        \item[(ii)] \cite[Proposition 3.4]{Rei03} The stratum $\operatorname{Rep}_{(d^i)}^{HN}(Q)$ is a locally closed subvariety in $\operatorname{Rep}_{d}(Q)$ of codimension $-\sum_{i<j} \lge d^i, d^j \rge$. Moreover, its closure equals $\operatorname{Rep}_{(d^i)}(Q)$.
    \end{itemize}
\end{theorem}

\subsection{Kirwan surjectivity for quiver varieties}\label{ss:kirwan}
Consider the $\mathbb G_m$-action on $\mathrm{Rep}_d(\bar Q)$ that rescales all the arrows equally. This action commutes with the $\GL_d$-action and restricts to the action on $\mu_d^{-1}(0)$, making the inclusion
$$
\mu_d^{-1}(0)\hookrightarrow\mathrm{Rep}_d(\bar Q)
$$
$\GL_d\times\mathbb G_m$-equivariant. Since the $\mathbb G_m$-action contracts both spaces to a point, we have the following diagram with the square of isomorphisms above:

\begin{center}
    \begin{tikzcd}
H_{\GL_d}^\bullet(\mathrm{pt}) \arrow[d, "\simeq"] & H_{\GL_d}^\bullet(\mathrm{pt}) \arrow[l, no head, equal] \arrow[d, "\simeq"] \\
H_{\GL_d}^\bullet(\mu^{-1}_d(0)) \arrow[d]         & H_{\GL_d}^\bullet(\mathrm{Rep}_d(\bar Q)) \arrow[d] \arrow[l,"\simeq"']                     \\
H^\bullet(\mc M_d^\chi(Q))                         & H^\bullet(R_d^\chi(\bar Q)) \arrow[l]                                            \fullstopbelow
\end{tikzcd}
\end{center}

We will be particularly interested in the situations when the map
\begin{equation}\label{eq:equivcohomotogit}
    H_{\GL_d}^\bullet(\mu^{-1}_d(0))\to H^\bullet(\mc M_d^\chi(Q))
\end{equation}
is surjective. The following result of McGerty and Nevins \cite{MGN18} gives us a wide variety of cases where it is:
\begin{theorem}\label{t:kirwansurjfornakajima}
    If $\chi$ is generic and $\mc M_d^\chi(Q)$ is a Nakajima quiver variety, i.e. $Q$ is the Crawley-Boevey quiver for another quiver $Q'$ with a nonzero framing $w$ and $d=(d',1)$, then the map (\ref{eq:equivcohomotogit}) is surjective.
\end{theorem}

\subsection{Kac Polynomials}
For any quiver $Q$ and dimension vector $d$, let $A_d(q)$ denote the number of absolutely indecomposable representations of $Q$ of dimension vector $d$ over a finite field with $q$ elements, where $q$ is a prime power. In \cite{Kac}, it was proven that this is a polynomial in $q$. We call $A_d(q)$ the \textbf{Kac polynomial}.

When $Q$ is a quiver without loops, the Kac polynomial contains significant information about the associated Kac-Moody Lie algebra $\g_Q$. Some key properties are summarized below:

\begin{theorem}\label{t:KacvsKM}
    Let $Q$ be a quiver without loops and let $d$ be a dimension vector. Let $\g_Q$ be the associated Kac-Moody Lie algebra. Then $A_d(q) \neq 0$ if and only if $d$ is a root of $\g_Q$. In this case, the following hold:
    \begin{itemize}
        \item[(i)] \cite[\S 1.13-1.15]{Kac} The degree of $A_d(q)$ equals $1 - \lge d, d \rge$. Moreover, $A_d(q)$ is independent of the orientation of $Q$.
        \item[(ii)] \cite[Corollary 1.6]{HLRV} The polynomial $A_d(q)$ has nonnegative coefficients.
        \item[(iii)] \cite[Section 3]{Hau} The evaluation $A_d(0)$ equals the multiplicity of $d$ in $\g_Q$.
    \end{itemize}
\end{theorem}

We are particularly interested in the cohomological interpretation of the coefficients of Kac polynomials:

\begin{theorem}[{\cite[Theorem 1.1]{CBVdB}}]\label{t:kacviacohomo}
    Let $d$ be an indivisible imaginary root for $Q$ and let $\chi$ be a generic stability parameter for $d$. Then
    \[
    A_d(q) = \sum_{i=0}^{1-\lge d, d \rge} \dim H^{2i}(\mc M_d(Q), \C) q^{1-\lge d, d \rge - i}.
    \]
\end{theorem}

\begin{remark}
{\rm
    This result was extended in \cite[Corollary 1.6]{HLRV} to all imaginary roots, where the coefficients of the corresponding Kac polynomials were interpreted as dimensions of certain isotypic components of cohomologies of certain quiver varieties. While we do not use this result in this paper, it would be natural to extend our results to the representation stabilization of these cohomologies.
}
\end{remark}

\section{Stabilization for Quiver Representations}
Let $Q = (Q_0, Q_1)$ be a quiver without edge loops. We assume that $Q$ is {\bfseries symmetric}, which means that for each $i,j\in Q$, the number of arrows $i\to j$ equals the number of arrows from $j\to i$. This assumption is equivalent to the Euler form being symmetric.

Let $\delta$ be a dimension vector for $Q$. Let $\{C_k\}_k$ be the connected components of $\supp\delta$. We assume that for any $i\in Q_0$ and any $k,l$
\begin{equation}\label{eq:condition}
  \lge e_i, \delta \rge < 0, \quad \lge \delta, e_i \rge < 0, \quad\operatorname{dist}(C_k,C_l)\le 1 \tag{$\ast$}.
\end{equation}
The reason for these assumptions will become clear in the following two lemmas. Note that the first two assumptions are equivalent to
\[
\sum_{i \to j} \delta_j > \delta_i, \quad \sum_{j \to i} \delta_j > \delta_i.
\]
We could also replace the Euler form with any other form in (\ref{eq:condition}). In particular, when we say that $\delta$ satisfies (\ref{eq:condition}) for the Cartan form of $Q$, we would mean the analogous condition applied to the Cartan form.


We fix another dimension vector $d$. We wish to show that for each integer $k \geq 0$, the dimension of the $k$-th cohomology of 
\[
\mathrm{Rep}_{d + n\delta}(Q) \git_\chi G
\]
stabilizes as $n \to \infty$ whenever $d+n\delta$ is indivisible. For each such $d+n\delta$, we choose a generic stability parameter for which the semistable locus is nonempty (its existence will follow from Corollary \ref{cor:codimofHNlarge}).

\begin{lemma}
    The quantity
    \[
    \lge d + n\delta, d + n\delta \rge = \lge d, d \rge + n\lge d, \delta \rge + n\lge \delta, d \rge + n^2\lge \delta, \delta \rge
    \]
    tends to $-\infty$ as $n \to \infty$. Thus, $d + n\delta$ becomes an imaginary root for large $n$.
\end{lemma}
\begin{proof}
    Follows directly from (\ref{eq:condition}).
\end{proof}

\begin{lemma}\label{l:formmaximum}
    Let $\delta$ be a dimension vector satisfying {\rm (\ref{eq:condition})} and assume that $\supp d=Q$. Then
    $$
    \lim_{n\to\infty}\max_v\lge v,d+n\delta-v\rge=-\infty,
    $$
    where the maximum is taken over all dimension vectors $0\le v\le d+n\delta$ such that $v\ne 0,d+n\delta$.
    \end{lemma}
\begin{proof}
    Denote $d+n\delta$ by $\tau$ for simplicity. Define
    \begin{equation}\label{eq:stabbound}
        M_n:=\max\left\{\max_{i\in Q}\lge e_i, \tau\rge(1-1/\tau_i),-\min_{i\in Q}\tau_i\min_{j\in\supp\delta}\tau_j\right\}.
    \end{equation}
    
    It is easy to see that $M_n<0$ and $\lim_{n\to\infty}M_n=-\infty$. We will show that for any dimension vector $0\le v\le \tau$ such that $v\ne 0,\tau$ we have the inequality
    $$
    \lge v,\tau-v\rge\le M_n,
    $$
    which would prove the lemma. 
    
    Let $E=(e_{ij})$ be the matrix for the Euler form. Recall that it is symmetric. We can write
\begin{align*}
    2\lge v,\tau-v\rge&=2\sum_{i,j} e_{ij}v_i(\tau_j-v_j)\\
    &=\sum_{i,j}e_{ij}v_i(\tau_j-v_j)+\sum_{i,j}e_{ij}v_j(\tau_i-v_i)\\
    &=\sum_{i,j}e_{ij}(v_i\tau_j+v_j\tau_i-2v_iv_j)\\
    &=\sum_{i,j}e_{ij}\tau_i\tau_j\left(\frac{v_i}{\tau_i}+\frac{v_j}{\tau_j}-2\frac{v_i}{\tau_i}\frac{v_j}{\tau_j}\right)\\
    &=\sum_{i,j}e_{ij}\tau_i\tau_j\left(\frac{v_i}{\tau_i}-\frac{v_i^2}{\tau_i^2}+\frac{v_j}{\tau_j}-\frac{v_j^2}{\tau_j^2}\right)+\sum_{i,j}e_{ij}\tau_i\tau_j\left(\frac{v_i}{\tau_i}-\frac{v_j}{\tau_j}\right)^2\\
    &=2\sum_{i,j}e_{ij}\tau_i\tau_j\left(\frac{v_i}{\tau_i}-\frac{v_i^2}{\tau_i^2}\right)+\sum_{i\ne j}e_{ij}\tau_i\tau_j\left(\frac{v_i}{\tau_i}-\frac{v_j}{\tau_j}\right)^2\\
    &=2\sum_{i}v_i\left(1-\frac{v_i}{\tau_i}\right)\lge e_i,\tau\rge+\sum_{i\ne j}e_{ij}\tau_i\tau_j\left(\frac{v_i}{\tau_i}-\frac{v_j}{\tau_j}\right)^2
\end{align*}

Since $e_{ij}\le 0$ whenever $i\ne j$,  the second sum is nonpositive. Assume that there is $k\in Q_0$ such that $1\le v_k\le \tau_k-1$. Then
$$
2\sum_{i}v_i\left(1-\frac{v_i}{\tau_i}\right)\lge e_i,\tau\rge\le 2v_k\left(1-\frac{v_k}{\tau_k}\right)\lge e_i,\tau\rge\le 2M_n,
$$
as desired.

Assume now that for all $i\in Q_0$, $v_i\in\{0,\tau_i\}$. Then we have a disjoint union $Q=A\sqcup B$, where $A:=\supp v$ and $B:=\supp (\tau-v)$. According to the computation above,
$$
\lge v,\tau-v\rge=\sum_{i\in A, j\in B}e_{ij}\tau_i\tau_j.
$$

Suppose that there are $k\in A$ and $l\in B$ such that there is an arrow between $k$ and $l$ and at least one of $k$ or $l$ lies in $\supp\delta$. Without loss of generality, let $k\in\supp \delta$. Then
$$
\sum_{i\in A, j\in B}e_{ij}\tau_i\tau_j\le e_{k,l} \tau_k\tau_l\le M_{n},
$$
proving the claim. 

Suppose the contrary now, i. e. that all edges between $A$ and $B$ have endpoints in $Q\setminus\supp\delta$. In particular, this implies that each connected component of $\supp\delta$ lies in $A$ or $B$. Suppose that there are such components $C_1$ and $C_2$ lying in $A$ and $B$, respectively. By (\ref{eq:condition}), there exists an edge from $C_1$ to $C_2$ or a vertex $v$ with edges to both $C_1$ and $C_2$. Both these cases produce an edge between $A$ and $B$ having an endpoint lies in $C_1$ or $C_2$, a contradiction.

Finally, suppose that $\supp\delta$ lie entirely in $A$ or $B$. Without loss of generality, let $\supp\delta\subset A$. Take any vertex $k\in B$. Then it follows from (\ref{eq:condition}) that $k$ is connected to at least one vertex in $\supp\delta$, a contradiction again. Thus, there must be an edge between $A$ and $B$ with at least one endpoint in $\supp\delta$, which finishes the proof.
\end{proof}


\begin{corollary}\label{cor:codimofHNlarge}
     For large enough $n$, the codimensions of all HN strata $\operatorname{Rep}_{(d^i)}^{HN}(Q)$ other than the semistable one in $\operatorname{Rep}_{d+n\delta}(Q)$ are large. In particular, the semistable locus is non-empty for any character $\chi$.
\end{corollary}

\begin{proof}
    Pick such stratum $\operatorname{Rep}_{(d^i)}^{HN}(Q)$. According to \ref{t:HNfiltandstrat}(ii), its closure equals $\operatorname{Rep}_{(d^i)}(Q)$, which by definition is contained in $\operatorname{Rep}_{(d^1,d+n\delta-d^1)}(Q)$. By (i) of the same theorem, the codimension of $\operatorname{Rep}_{(d^1,d+n\delta-d^1)}(Q)$ is exactly $-\lge d^1,d+n\delta-d^1\rge$. Lemma \ref{l:formmaximum} then concludes the claim.
\end{proof}

Because of this corollary, we no longer need to worry about nonemptiness of the semistable locus. The only property of the character $\chi$ that we will be choosing is the following:
\begin{definition}
    Let $d$ be a dimension vector for a quiver $Q$. A character $\chi$ of $\GL_d$ is called {\bfseries generic} with respect to $d$ if for every positive root $0\le v\le d$ with $v\ne 0,d$ we have $s(v)\ne s(d)$, where $s$ is the corresponding slope.
\end{definition}

\begin{lemma}\
    \begin{itemize}
        \item[(i)] If $\chi$ is generic for $d$, then $\chi$-semistability is equivalent to $\chi$-stability.
        \item[(ii)] Such $\chi$ exists if and only if $d$ is indivisible.
    \end{itemize}
\end{lemma}
\begin{proof}
    If $\chi$ is $d$-generic, then the inequality in Definition \ref{def:stable} is always strict. This implies (i).

    For (ii), assume first that $d$ is divisible. Then $d=ad'$ for $a\in\Z_{\ge 2}$ and another dimension vector $d'$. Then $s(d)=s(d')$ for any character $\chi$, proving that there is no $d$-generic one.

    Assume now that $d$ is indivisible. Note that $s(d)=s(v)$ is a linear equation on $\chi$. If all such equations are nontrivial, then they give finitely many hyperplanes in $\Z^{Q_0}$, and any $\chi$ from the complement would be $d$-generic. Assume that one of such equations is trivial, say for $v$. It gives
    $$
    \sum\chi_i\left(\frac{v_i}{|v|}-\frac{d_i}{|d|}\right)=0,
    $$
    which implies that $d_i\cdot |d|/|v|=v_i\in\Z$ for each $i$. This contradicts the indivisibility of $d$.
\end{proof}

\begin{theorem}\label{t:cohomostab}
    Let $Q$ be a symmetric quiver, and let $d, \delta$ be dimension vectors such that $\delta$ satisfies {\rm (\ref{eq:condition})}. Then for each $k \geq 0$, the dimension of
    \[
    H^{2k}\bigl(R_{d+n\delta}(Q), \C\bigr)
    \]
    is eventually constant for all sufficiently large $n$ such that $d + n\delta$ is indivisible, while all odd cohomologies are $0$. If we denote the stabilized dimensions by $h^{2k}_{\operatorname{st}}(d, \delta)$, then:
    \begin{equation}\label{eq:convforsymmquiv}
        \sum_{i\geq 0} h^{2i}_{\operatorname{st}}(d, \delta) q^i = \frac{(1-q) p^{|\operatorname{supp} \delta|}(q)}{\prod_{i \notin \operatorname{supp} \delta} \prod_{k=1}^{d_i} (1 - q^{k})},
    \end{equation}
    where $p^l(q)$ denotes the generating function of the partition number into $l$ colors, i.e., $p^l(q) = p(q)^l$, where $p(q)$ is the generating function for the ordinary partition number.
\end{theorem}

\begin{proof}

    
    We prove this by counting points over finite fields, motivated by ideas in \cite[Section 6.6]{Rei08}. Write
    \begin{align*}
        |R_{d+n\delta}(Q)(\F_{q})|&=\frac{(q-1)|\mathrm{Rep}_{d+n\delta}(Q)^{ss}(\mathbb{F}_q)|}{|G(\mathbb{F}_q)|}\\
        &=\frac{(q-1)|\mathrm{Rep}_{d+n\delta}(Q)(\mathbb{F}_q)|}{|G(\mathbb{F}_q)|}-\sum_{(d^i)\ne (d+n\delta)}\frac{(q-1)|\mathrm{Rep}_{(d^i)}^{HN}(Q)(\mathbb{F}_q)|}{|G(\mathbb{F}_q)|}.
    \end{align*}
    
    By Corollary \ref{cor:codimofHNlarge}, the codimensions of the HN strata tend to $\infty$ as $n\to\infty$. Choose $M>0$ and $n\ge M$ such that the strata have codimension at least $M$. Then
    \begin{align*}
        |R_{d+n\delta}(Q)(\F_{q})|&=\frac{(q-1)|\mathrm{Rep}_{d+n\delta}(Q)(\mathbb{F}_q)|}{|G(\mathbb{F}_q)|}+O(q^{\dim R_{d+n\delta}(Q)-M})\\
        &=\frac{(q-1)q^{\sum_{i\to j}(d_i+n\delta_i)(d_j+n\delta_j)}}{\prod_{i\in Q_0}q^{(d_i+n\delta_i)^2}\prod_{k=1}^{d_i+n\delta_i}(1-q^{-k})}+O(q^{1-\lge d,d\rge-M})\\
        &=\frac{(q-1)q^{-\lge d,d\rge}}{\prod_{i\in Q_0}\prod_{k=1}^{d_i+n\delta_i}(1-q^{-k})}+O(q^{1-\lge d,d\rge-M})\\
        &=\frac{(q-1)q^{-\lge d,d\rge}}{\prod_{i\in \operatorname{supp}\delta}\prod_{k=1}^{d_i+n\delta_i}(1-q^{-k})}\cdot \frac{1}{\prod_{i\notin \operatorname{supp}\delta}\prod_{k=1}^{d_i}(1-q^{-k})}+O(q^{1-\lge d,d\rge-M})\\
        &=\frac{(q-1)q^{-\lge d,d\rge}}{\prod_{i\in \operatorname{supp}\delta}\prod_{k=1}^{\infty}(1-q^{-k})}\cdot \frac{1}{\prod_{i\notin \operatorname{supp}\delta}\prod_{k=1}^{d_i}(1-q^{-k})}+O(q^{1-\lge d,d\rge-M})\\
        &=\frac{(q-1)q^{-\lge d,d\rge} p^{|\operatorname{supp} \delta|}(q^{-1})}{\prod_{i \notin \operatorname{supp} \delta} \prod_{k=1}^{d_i} (1 - q^{-k})}+O(q^{1-\lge d,d\rge-M}).
    \end{align*}
    
    We see that the top $M$ powers here are exactly the first $M$ coefficients of the power series (\ref{eq:convforsymmquiv}).    Adapting the reasoning in \cite[2.4]{CBVdB} and \cite[Lemma A.1]{CBVdB}, we conclude that $R_{d+n\delta}(Q)(\F_q)$ is pure and the numbers $\dim H_c^{2(1-\lge d,d\rge-i)}(R_{d+n\delta}(Q),\C)$ for $i=0,\ldots, M-1$ are the first $M$ coefficients of (\ref{eq:convforsymmquiv}). Poincar\'{e} duality yields
    $$
    \dim H_c^{2(1-\lge d,d\rge-i)}(R_{d+n\delta}(Q),\C)=\dim H^{2i}(R_{d+n\delta}(Q),\C),
    $$
    which gives the desired result as we let $M\to\infty$.

\end{proof}

\begin{remark}
    {\rm
        Unlike in \cite[2.4]{CBVdB}, we did not prove that $|R_{d+n\delta}(Q)(\F_q)|$ is a polynomial in $q$. This is true though, see \cite[Theorem 6.7]{Rei03}. What was enough for us is that it is asymptotically a power series in $q$ as $n\to\infty$.
    }
\end{remark}

\begin{remark}\label{r:stabbound}
{\rm
    The proof of Lemma \ref{l:formmaximum} actually gives $n$ for which $H^{2k}\bigl(R_{d+n\delta}(Q), \C\bigr)$ stabilizes. Indeed, such stabilization occurs whenever the codimensions of the non-semistable HN strata become larger then $k$, which occurs when $M_n<-k$, where $M_n$ is defined in \ref{eq:stabbound}. 
}
\end{remark}

\begin{remark}
{\rm
    We may relax our assumption for $Q$ to be symmetric: it is enough to assume that after removing some arrows of $Q$, it becomes a symmetric quiver satisfying the assumptions of Theorem \ref{t:cohomostab}. The same proof applies since the Euler form of the new quiver is bounded below by the Euler form of $Q$.
}
\end{remark}

There is another explanation for where the formula (\ref{eq:convforsymmquiv}) comes from. Recall the map
$$
H^\bullet_{\GL_d}(\mathrm{pt})\to H^\bullet(R_{d}(Q)).
$$
from Subsection \ref{ss:kirwan}. Since $\GL_d$ acts on $\mathrm{Rep}_d(Q)$ with stabilizer $\mathbb G_m$, this map factors through
\begin{equation}\label{eq:kirwanmaptorep}
    H^\bullet_{\GL_d/\mathbb G_m}(\mathrm{pt})\to H^\bullet(R_{d}(Q)).
\end{equation}

Let $\g_m$ be the Lie algebra of $\mathbb G_m$. We have
\begin{align*}
    H^\bullet_{\GL_d}(\mathrm{pt})=S(\mathfrak{gl}_d^*)^{\GL_d}=S((\mathfrak{gl}_d/\g_m)^*)^{\GL_d}\otimes S(\g_m)=H^\bullet_{\GL_d/\mathbb G_m}(\mathrm{pt})\otimes \C[t],
\end{align*}
where $\deg t=2$. We compute
\begin{align*}
    H^\bullet_{\GL_d}(\mathrm{pt})&=S(\mathfrak{gl}_d^*)^{\GL_d}\\
    &=\bigotimes_{i\in Q_0}S(\mathfrak{gl}_{d_i}^*)^{\GL_{d_i}}\\
    &=\bigotimes_{i\in Q_0}\C[P_{i,1},\ldots, P_{i,d_i}],
\end{align*}
where $\deg P_{i,j}=j$. Thus, the Poincar\'e series for $H^\bullet_{\GL_d/\mathbb G_m}(\mathrm{pt})$ equals
$$
(1-q)\cdot\prod_{i\in Q_0}\frac{1}{\prod_{j=1}^{d_i}(1-q^j)},
$$
where $\deg q=2$. If we put $d+n\delta$ in this formula and let $n\to\infty$, we will get the right-hand side of (\ref{eq:convforsymmquiv}). Since the nontrivial HN strata have large codimension for large $n$, the map (\ref{eq:kirwanmaptorep}) becomes an isomorphism in low degrees for large $n$. This gives a geometric proof of Theorem \ref{t:cohomostab} not involving point counting. 

\section{Stabilization for Kac polynomials}
In this section, we address stabilization of Kac polynomials. We introduce the following notation. For a positive integer $n$, set
$$
\phi_n(q):=\prod_{i=1}^n(1-q^i),
$$
and extend this definition to a dimension vector $d=(d_1,\ldots,d_n)$ by
$$
\phi_d(q):=\prod_{i=1}^n\phi_{d_i}(q).
$$
Let $\mc P_n$ denotes the set of all partitions of $n$ and $\mc P:=\bigsqcup_n\mc P_n$. For $\pi\in \mc P_n$ written as $(1^{r_1},\ldots, n^{r_n}$, set
$$
b_\pi(q):=\prod_{i}\phi_{r_i}(q).
$$
For $\pi_1,\pi_2\in\mc P$, let $\pi_1'$ and $\pi_2'$ be the corresponding dual partitions. We define
$$
\lge\pi_1,\pi_2\rge:=\sum_i (\pi_1')_i(\pi_2')_i.
$$

In \cite{Hua}, Jiuzhao Hua defines the following generating function:
\begin{align*}
    P_{\Gamma}(X_1, \dots, X_n, q) &= \sum_{\pi^1, \dots, \pi^n \in \mathcal{P}} 
\frac{
\prod_{1 \leq i \leq j \leq n} q^{a_{ij} \langle \pi^i, \pi^j \rangle}
}{
\prod_{1 \leq i \leq n} q^{\langle \pi^i, \pi^i \rangle} b_{\pi^i}(q^{-1})
}
X_1^{|\pi^1|} \cdots X_n^{|\pi^n|}\\
&= \sum_{\pi^1, \dots, \pi^n \in \mathcal{P}} 
\frac{
q^{\sum_{1 \leq i \leq j \leq n}a_{ij} \langle \pi^i, \pi^j \rangle-\sum_{1 \leq i \leq n}\langle \pi^i, \pi^i \rangle}
}{
\prod_{1 \leq i \leq n} b_{\pi^i}(q^{-1})
}
X_1^{|\pi^1|} \cdots X_n^{|\pi^n|},
\end{align*}
where $a_{ij}$ is the number of edges between $i$ and $j$ in the underlying graph of $Q$. Replacing all the partitions by their duals in the above formula, we can rewrite it as
\begin{align*}
    P_{\Gamma}(X_1, \dots, X_n, q) &= \sum_{\pi^1, \dots, \pi^n \in \mathcal{P}} 
\frac{
q^{\sum_{k\ge 1}\sum_{1 \leq i \leq j \leq n}a_{ij} \pi^i_k\pi^j_k -\sum_{1 \leq i \leq n}\pi^i_k\pi^i_k }
}{
\prod_{1 \leq i \leq n} b_{(\pi^i)'}(q^{-1})
}
X_1^{|\pi^1|} \cdots X_n^{|\pi^n|}.
\end{align*}

In the notation of the above formula, denote $d^k:=(\pi^1_k,\ldots, \pi^n_k)$, which we will consider dimension vectors for $Q$. Then
$$
\sum_{1 \leq i \leq j \leq n}a_{ij} \pi^i_k\pi^j_k -\sum_{1 \leq i \leq n}\pi^i_k\pi^i_k=-\lge d^k,d^k\rge,
$$
where we are considering the standard Euler form for $Q$. Moreover, $X_1^{|\pi^1|} \cdots X_n^{|\pi^n|}=X^{\sum_{k\ge 0} d^k}$, where for a dimension vector $d=(d_1,\ldots,d_n)$ we set
$$
X^d:=\prod X_i^{d_i}.
$$
To rewrite $\prod_{1 \leq i \leq n} b_{(\pi^i)'}(q^{-1})$, notice that 
$$
b_\pi(q)=\prod_{i}\prod_{j=1}^{\pi_i'-\pi_{i+1}'}(1-q^j).
$$
This gives
\begin{align*}
    \prod_{i=1}^n b_{(\pi^i)'}(q^{-1})&=\prod_{i=1}^n\prod_{k}\prod_{j=1}^{\pi_k^i-\pi_{k+1}^i}(1-q^{-j})=\prod_{k}\phi_{d^k-d^{k+1}}(q^{-1}).
\end{align*}

Combining all the above computations, we get
\begin{equation*}
    P_{\Gamma}(X_1, \dots, X_n, q) = \sum_d\sum_{\substack{d=d^1+\ldots +d^s\\d^1\ge\ldots\ge d^s}} 
\frac{
q^{-\sum_{k}\lge d^k,d^k\rge}
}{
\prod_k\phi_{d^k-d^{k+1}}(q^{-1})
}
X^d.
\end{equation*}

Then we have the following particular case of \cite[Theorem 4.6]{Hua}:
\begin{theorem}
    Let $d$ be an indivisible vector. Then $A_d(q)$ equals the $X^d$-coefficient of the generating function $(q-1)\log(P_{\Gamma}(X_1, \dots, X_n, q))$. Equivalently,
    \begin{equation}\label{eq:formulaforKac}
        A_d(q)=(q-1)\sum_{l=1}^\infty\frac{(-1)^{l+1}}{l}\sum_{\substack{(\alpha^1,\ldots,\alpha^l)\\d=\alpha^1+\ldots+\alpha^l}}\prod_{i=1}^l\left[\sum_{\substack{\alpha^i=d^{i,1}+\ldots +d^{i,s}\\d^{i,1}\ge\ldots\ge d^{i,s}}} 
\frac{
q^{-\sum_{k}\lge d^{i,k},d^{i,k}\rge}
}{
\prod_k\phi_{d^{i,k}-d^{i,k+1}}(q^{-1})
}\right].
    \end{equation}
\end{theorem}

Let $\delta$ be an imaginary root for $Q$ satisfying (\ref{eq:condition}) for the Cartan form of $Q$. Write
\[
    A_d(q) = \sum_{i=0}^{1-\lge d, d \rge} a_i^{(d)} q^{1-\lge d, d \rge - i}.
\]

\begin{lemma}\label{l:formmaxwithanydecomp}
    Let $\delta$ be a dimension vector satisfying {\rm (\ref{eq:condition})} for the Cartan form of $Q$ and assume that $\supp d=Q$. Then, for large enough $n$
    \begin{align*}
        \max_{d+n\delta=d^1+\ldots +d^l,\,l\ge 2}\left(\lge d+n\delta,d+n\delta\rge-\sum_k\lge d^k,d^k\rge\right)&=\max_{d+n\delta=d^1+d^2}\left(\lge d+n\delta,d+n\delta\rge-\sum_k\lge d^k,d^k\rge\right)\\
        &=\max_{0<v<d}2\lge v,d+n\delta-v\rge.
    \end{align*}
    In particular,
    $$
    \lim_{n\to\infty}\max_{d+n\delta=d^1+\ldots +d^l,\,l\ge 2}\left(\lge d+n\delta,d+n\delta\rge-\sum_k\lge d^k,d^k\rge\right)=-\infty.
    $$
\end{lemma}
\begin{proof}
    Write $d+n\delta=d^1+\ldots+d^l$. According to Lemma \ref{l:formmaximum}, there exists $N>0$ such that for every $n>N$
    $$
    \max_{0<v<d}2\lge v,d+n\delta-v\rge<0.
    $$
    In particular, 
    $$
    \lge d^1,d^2+\ldots+d^{l}\rge<0,
    $$
    which implies that there exists $s\ne 1$ such that $\lge d^1,d^s\rge<0$. Without loss of generality, assume that $s=2$. Then
    $$
    -\lge d^1,d^1\rge-\lge d^2,d^2\rge=-\lge d^1+d^2, d^1+d^2\rge+2\lge d^1,d^2\rge<-\lge d^1+d^2, d^1+d^2\rge.
    $$
    
    This implies that replacing the pair $(d^1,d^2)$ by one element $d^1+d^2$ increases the quantity
    $$
    \lge d+n\delta,d+n\delta\rge-\sum_k\lge d^k,d^k\rge
    $$
    and reduces the number of summands by $1$, so the claim follows.
\end{proof}
\begin{theorem}\label{t:kacstab}
    For every $i$ the coefficients $a_i^{(d+n\delta)}$ for indivisible $d+n\delta$ stabilize for sufficiently large $n$. Moreover, the generating function of the stabilized coefficients is given by:
    \[
    \frac{(1-q) p^{|\operatorname{supp} \delta|}(q)}{\prod_{i \notin \operatorname{supp} \delta} \prod_{k=1}^{d_i} (1 - q^k)}.
    \]
\end{theorem}

\begin{proof}
    Take one particular summand for given $l$, $\alpha^i$, and $d^{i,j}$ in \ref{eq:formulaforKac}. Its asymptotics are
    $$
    O\left(q^{-\sum_{i,k}\lge d^{i,k},d^{i,k}\rge }\right).
    $$
    Since $\sum_{i,k}d^{i,k}=d+n\delta$, we conclude from Lemma \ref{l:formmaxwithanydecomp} that for a given $M>0$ and large $n$, all such summand with at least two $d^{i,k}$ does not affect the first $M$ coefficients. Thus, in the limit $M,n\to\infty$, we obtain
    \begin{align*}
        \lim_{n\to\infty}q^{1-\lge d,d\rge}A_{d+n\delta}(q^{-1})&=\lim_{n\to\infty}(1-q)\cdot\frac{1}{\phi_{d+n\delta}(q)}\\
        &=\lim_{n\to\infty}\frac{(1-q)}{\prod_{i}\prod_{k=1}^{d_i+n\delta_i}(1-q^{k})}\\
        &=\frac{(1-q) p^{|\operatorname{supp} \delta|}(q)}{\prod_{i \notin \operatorname{supp} \delta} \prod_{k=1}^{d_i} (1 - q^k)},
    \end{align*}
as desired.
\end{proof}

\begin{remark}\label{r:geometrickacstabinterpr}
    {\rm
    Note the following interpretation of this result. If $\delta$ satisfies (\ref{eq:condition}) for the Cartan form, then $\delta$ satisfies (\ref{eq:condition}) for the Euler form of the double quiver $\bar Q$. Then the combination of Theorems \ref{t:cohomostab} and \ref{t:kacstab} tells that for any $k$ and large enough $n$, the closed inclusion
    \begin{equation}\label{eq:inclofmintorep}
        \mc M^\chi_{d+n\delta}(Q)\subset R_{d+n\delta}^\chi(Q)
    \end{equation}
    induces an isomorphism on the first $k$ cohomologies.
    }
\end{remark}

\section{Relaxing the (\ref{eq:condition}) condition}
In this section, we consider the {\bfseries weak (\ref{eq:condition}) condition}, which is (\ref{eq:condition}) with all strict inequalities replaced by $\le$. We will show the stabilization of Kac polynomials in this case too.

Since the proof is going to be technical and based on operations with formula (\ref{eq:formulaforKac}), let us sketch its main ideas. Recall that in the proof of Theorem \ref{t:kacstab}, we were analyzing each summand in (\ref{eq:formulaforKac}) corresponding to a choice of the decompositions $d=\alpha^1+\ldots+\alpha^l$ and $\alpha^i=d^{i,1}+\ldots +d^{i,s}$ for each $i$. It turned out that if the resulting partition $d=\sum_{i,j} d^{i,j}$ was nontrivial, then the $q$-degree of that summand was much smaller than the top $q$-degree, which meant that only the summand with trivial partitions mattered in the limit.

When we use the weak (\ref{eq:condition}) condition, this statement is no longer true. However, we notice that the set of partitions giving summands of high $q$-degree remains the same if we forget the asymptotically largest element in each of these partitions (see Lemma \ref{l:formmaxwithrelaxedstar}). Thus, in the limit, the set of partitions that matter for the top coefficients remains the same, showing stabilization. However, because of the presence of extra terms, this proof does not give a formula for the stabilized coefficients. Nevertheless, we hope that the formula is the same (see Conjecture \ref{conj:exactlimit}).

For two dimension vectors $d$ and $e$, we will write $d\ge e$ if $d-e$ is a dimension vector and $d>e$ if $d-e$ is a nonzero dimension vector.

\begin{lemma}\label{l:formmaxwithrelaxedstar}
    Choose $M>0$ and $\varepsilon>0$ and assume that $n$ is large enough. Let $\delta$ be a dimension vector satisfying relaxed (\ref{eq:condition}) for the Cartan form of $Q$ and assume that $(d+n\delta,e_i)<0$ for any $i\in Q_0$. Decompose $d+n\delta=d^1+\ldots+d^k$ into the sum of dimension vectors. If 
    $$
    \lge d+n\delta,d+n\delta\rge-\sum_k\lge d^k,d^k\rge>-M,
    $$
    then one of the $d^j$ satisfies
    $$
    d^j-(n-\varepsilon\sqrt n)\delta\ge 0.
    $$
    Moreover, $(d-d^j,\delta)=0$.
\end{lemma}
\begin{proof}
    The proof of Lemma \ref{l:formmaxwithanydecomp} applies in this case, so we may assume that the decomposition in our case is 
    $$
    \tau:=d+n\delta=d^1+d^2=v+(\tau-v).
    $$
    Thus,
    $$
    \lge d+n\delta,d+n\delta\rge-\lge v,v\rge>-\lge \tau-v,\tau-v\rge>-M,
    $$
    which is equivalent to
    $$
    \lge v,\tau-v\rge>-M.
    $$
    
    Let $(c_{ij})$ denote the entries of the Cartan matric for $Q$. Recall the formula from the proof of Lemma \ref{l:formmaximum}:
    $$
-2M<2\lge v,\tau-v\rge=2\sum_{i}v_i\left(1-\frac{v_i}{\tau_i}\right)( e_i,\tau)+\sum_{i\ne j}c_{ij}\tau_i\tau_j\left(\frac{v_i}{\tau_i}-\frac{v_j}{\tau_j}\right)^2.
    $$
    Using our assumptions that $( e_i,\tau)<0$ and $c_{i,j}\le 0$ for $i\ne j$, we obtain
    $$
M>\sum_{i\in\supp\delta}v_i\left(1-\frac{v_i}{\tau_i}\right)=\sum_{i\in\supp\delta}v_i\left(1-\frac{v_i}{\tau_i+n\delta_i}\right).
    $$
Assume that for some $i\in\supp\delta$, $\varepsilon\sqrt n<v_i<\tau_i-\varepsilon\sqrt n$. Then
$$
M>v_i\left(1-\frac{v_i}{\tau_i+n\delta_i}\right)>\varepsilon\sqrt n\left(1-\frac{\varepsilon\sqrt n}{d_i+n\delta_i}\right)=\varepsilon\sqrt n-\frac{\varepsilon^2}{d_i/\sqrt n+\delta_i}.
$$
The right-hand side approaches $\infty$ as $n\to\infty$, a contradiction. Thus, for each $i\in\supp\delta$, $v_i<\varepsilon\sqrt n$ or $\tau_i-v_i<\varepsilon\sqrt n$.

Define 
\begin{align*}
    A&:=\{i\in\supp\delta:v_i<\varepsilon\sqrt n\}\\
    B&:=\{i\in\supp\delta:\tau_i-v_i<\varepsilon\sqrt n\}.
\end{align*}
As we just proved, $\supp\delta=A\sqcup B$. Our goal is to prove that one of $A$ and $B$ is empty, which would prove the first statement of the Lemma.

Suppose the contrary, i.e. that both $A$ and $B$ are nonempty. According to (\ref{eq:condition}), we have two cases:
\begin{enumerate}
    \item There are $i\in A$ and $j\in B$ such that $i$ and $j$ are connected in $Q$.
    \item There are $i\in A$, $j\in B$, $k\in Q_0$ such that $i$ and $j$ are both connected to $k$.
\end{enumerate}
Assume the first case. Then
\begin{align*}
    -2M<c_{ij}\tau_i\tau_j\left(\frac{v_i}{\tau_i}-\frac{v_j}{\tau_j}\right)^2&<c_{ij}\tau_i\tau_j\left(\frac{\varepsilon\sqrt n}{d_i+n\delta_i}-\frac{d_j+n\delta_j-\varepsilon\sqrt n}{d_j+n\delta_j}\right)^2\\
    &=c_{ij}(d_i+n\delta_i)(d_j+n\delta_j)\left(\frac{\varepsilon\sqrt n}{d_i+n\delta_i}+\frac{\varepsilon\sqrt n}{d_i+n\delta_i}-1\right)^2,
\end{align*}
which tends to $-\infty$ as $n\to\infty$, a contradiction. In the second case, using the same argument for the pairs $(i,k)$ and $(k,j)$, we get that $v_k=o(\sqrt n)$ and $n-v_k=o(\sqrt n)$, a contradiction.

To prove the last statement, observe that
$$
-M<(v,\tau-v)=(d,\tau-v)+n(\delta,\tau-v)-(\tau-v,\tau-v).
$$
Now, since $\tau-v\le \varepsilon\sqrt n$, we have the following:
$$
(d,\tau-v)=O(\sqrt n)\cdot (d,\delta),\qquad (\tau-v,\tau-v)=O(\varepsilon^2 n(\delta,\delta)).
$$
Since $\varepsilon$ can be as small as possible, we conclude that 
$$
-M<(v,\tau-v)\approx n(\delta,\tau-v).
$$
Since $(\delta,e_i)\le 0$, then $(\delta,\tau-v)\le 0$, which implies that $(\delta,\tau-v)=0$ for large $n$.
\end{proof}

\begin{theorem}\label{t:kacstabwithrelaxedstar}
    Let $d,\delta$ be dimension vectors such that $\delta$ satisfies weak (\ref{eq:condition}) and $(d,e_i)<0$ for any $i\in Q_0$. Then $h(q):=
    \lim_{n\to\infty}q^{-1+\lge d,d\rge}A_{d+n\delta}(q^{-1})$ exists, where the limit is taken over $n$ such that $d+n\delta$ is indivisible. Moreover, if Kirwan surjectivity holds for $Q$ and infinitely many indivisible dimension vectors of the form $d+n\delta$, then $h(q)$ is bounded above by
    \[
    \frac{(1-q) p^{|\operatorname{supp} \delta|}(q)}{\prod_{i \notin \operatorname{supp} \delta} \prod_{k=1}^{d_i} (1 - q^k)}.
    \]
\end{theorem}
\begin{proof}
    Choose $M>0$. We will utilize the highest $M$ powers in the equality \ref{eq:formulaforKac}. According to Lemma \ref{l:formmaxwithrelaxedstar}, for any such term there exists $s$ such that $d^{s,1}\ge (n-\sqrt n)\delta$. In particular, $\alpha^s\ge(n-\sqrt n)\delta$, and $s$ is the unique index with such property since
    $$
    2(n-\sqrt n)\delta >\frac32n\delta,
    $$
    which is not less than $d+n\delta$ for large $n$. Thus, we can write
    \begin{align*}
B_{d+n\delta}(q)&:=\frac{1}{1-q^{-1}}A_{d+n\delta}(q^{-1})\\
&=\sum_{l=1}^\infty\frac{(-1)^{l+1}}{l}\sum_{\substack{(\alpha^1,\ldots,\alpha^l)\\d=\alpha^1+\ldots+\alpha^l}}\prod_{i=1}^l\left[\sum_{\substack{\alpha^i=d^{i,1}+\ldots +d^{i,s}\\d^{i,1}\ge\ldots\ge d^{i,s}}} 
\frac{
q^{\sum_{k}\lge d^{i,k},d^{i,k}\rge}
}{
\prod_k\phi_{d^{i,k}-d^{i,k+1}}(q)
}\right]\\
&=\sum_{l=1}^\infty\frac{(-1)^{l+1}}{l}\sum_{s=1}^l\sum_{\substack{(\alpha^1,\ldots,\alpha^l)\\d=\alpha^1+\ldots+\alpha^l\\\alpha^s>(n-\sqrt n)\delta}}\prod_{i=1}^l\left[\sum_{\substack{\alpha^i=d^{i,1}+\ldots +d^{i,s}\\d^{i,1}\ge\ldots\ge d^{i,s}}}
\frac{
q^{\sum_{k}\lge d^{i,k},d^{i,k}\rge}
}{
\prod_k\phi_{d^{i,k}-d^{i,k+1}}(q)
}\right]+O(q^{\lge d,d\rge+M})\\
&=\sum_{l=1}^\infty(-1)^{l+1}\sum_{\alpha^1>(n-\sqrt n)\delta}\sum_{\substack{(\alpha^2,\ldots,\alpha^l)\\d-\alpha^1=\alpha^2+\ldots+\alpha^l}}\prod_{i=1}^l\left[\sum_{\substack{\alpha^i=d^{i,1}+\ldots +d^{i,s}\\d^{i,1}\ge\ldots\ge d^{i,s}}} 
\frac{
q^{\sum_{k}\lge d^{i,k},d^{i,k}\rge}
}{
\prod_k\phi_{d^{i,k}-d^{i,k+1}}(q)
}\right]+O(q^{\lge d,d\rge+M})\\
&=\sum_{l=1}^\infty(-1)^{l+1}\sum_{\alpha^1>(n-\sqrt n)\delta}f(d-\alpha^1)\left[\sum_{\substack{\alpha^1=d^{1,1}+\ldots +d^{1,s}\\d^{1,1}\ge\ldots\ge d^{1,s}}} 
\frac{
q^{\sum_{k}\lge d^{1,k},d^{1,k}\rge}
}{
\prod_k\phi_{d^{1,k}-d^{1,k+1}}(q)
}\right]+O(q^{\lge d,d\rge+M})\\
&=\sum_{l=1}^\infty(-1)^{l+1}\sum_{\substack{\alpha^1\ge d^{1,1}>(n-\sqrt n)\delta\\d^{1,1}\ge d^{1,2}}}\frac{q^{\lge d^{1,1},d^{1,1}\rge}}{\phi(d^{1,1}-d^{1,2})} f(d-\alpha^1)g(\alpha^1-d^{1,1},d^{1,2})+O(q^{\lge d,d\rge+M}),
    \end{align*}
where we denoted
\begin{align*}
    f(d-\alpha^1)&:= \sum_{\substack{(\alpha^2,\ldots,\alpha^l)\\d-\alpha^1=\alpha^2+\ldots+\alpha^l}}\prod_{i=2}^l\left[\sum_{\substack{\alpha^i=d^{i,1}+\ldots +d^{i,s}\\d^{i,1}\ge\ldots\ge d^{i,s}}} 
\frac{
q^{\sum_{k}\lge d^{i,k},d^{i,k}\rge}
}{
\prod_k\phi_{d^{i,k}-d^{i,k+1}}(q)
}\right]\\
g(\alpha^1-d^{1,1},d^{1,2})&:=\sum_{\substack{\alpha^1-d^{1,1}-d^{1,2}=d^{1,3}+\ldots +d^{1,s}\\d^{1,2}\ge\ldots\ge d^{1,s}}} 
\frac{
q^{\sum_{k\ge 2}\lge d^{1,k},d^{1,k}\rge}
}{
\prod_{k\ge 2}\phi_{d^{1,k}-d^{1,k+1}}(q)
}
\end{align*}

Since $\alpha^1>(n-\sqrt n)\delta$ and $d^{1,1}>(n-\sqrt n)\delta$ in our summation, $\alpha^1-\delta>0$ and $d^{1,1}-\delta>0$. We can easily see from the expression above that there is a bijection between the summands for $B_{d+n\delta}$ and $B_{d+(n-1)\delta}$ in the expression above sending $\alpha^1,\delta^{1,1}$ to $\alpha^1-\delta,d^{1,1}-\delta$, up to an element of degree $q^M$. Thus, the difference

\begin{align*}
    D_{d+n\delta}(q)&:=q^{-\lge d,d\rge}B_{d+n\delta}(q)-q^{-\lge d-\delta,d-\delta\rge}B_{d+(n-1)\delta}(q)\\
    &=q^{-\lge d,d\rge}\left(B_{d+n\delta}(q)-q^{(d,\delta)-\lge\delta,\delta\rge}B_{d+(n-1)\delta}(q)\right)\\
   &=q^{-\lge d,d\rge}\sum_{l=1}^\infty(-1)^{l+1}\sum_{\substack{\alpha^1\ge d^{1,1}>(n-\sqrt n)\delta\\d^{1,1}\ge d^{1,2}}}q^{\lge d^{1,1},d^{1,1}\rge}\left[\frac{1}{\phi_{d^{1,1}-d^{1,2}}(q)}-\frac{q^{(d-d^{1,1},\delta)}}{\phi_{d^{1,1}-d^{1,2}-\delta}(q)}\right]\cdot\\
   &\cdot f(d-\alpha^1)g(\alpha^1-d^{1,1},d^{1,2})+O(q^{M}).
\end{align*}
Note the appearance of the Cartan form in the second line and beyond. By Lemma \ref{l:formmaxwithrelaxedstar}, it is enough to restrict our sum to $d^{1,1}$ such that $(d+n\delta-d^{1,1},\delta)=0$. In this case,
\begin{align*}
    \frac{1}{\phi_{d^{1,1}-d^{1,2}}(q)}-\frac{q^{(d-d^{1,1},\delta)}}{\phi_{d^{1,1}-d^{1,2}-\delta}(q)}=\frac{1}{\phi_{d^{1,1}-d^{1,2}-\delta}(q)}\left[\frac{1}{\prod_{i\in \supp\delta}\prod_{j={d^{1,1}_i-d^{1,2}_i-\delta_i}}^{d^{1,1}_i-d^{1,2}_i}(1-q^j)}-1\right].
\end{align*}
Since $d^{1,1}>(n-\sqrt n)\delta$ and $d^{1,1}+d^{2,2}\le d+n\delta<\tfrac32 n\delta$, all nonzero powers of $q$ in the denominator inside the brackets are asymptotically greater than $\frac n2$. Therefore,
$$
\frac{1}{\phi_{d^{1,1}-d^{1,2}}(q)}-\frac{q^{(d-d^{1,1},\delta)}}{\phi_{d^{1,1}-d^{1,2}-\delta}(q)}=O(q^{n/2}),
$$
which shows that $D_{d+n\delta}(q)=O(q^M)$. This finishes the proof of stabilization. 

Since the generating function given in the second statement is the limit generating function for the equivariant cohomology of a point, we are done.
\end{proof}

It will be interesting to compute the stabilized coefficients. Computer computations for the quiver in Example \ref{e:hyperbolicquiver} support the following conjecture:
\begin{conjecture}\label{conj:exactlimit}
    With the assumptions of Theorem \ref{t:kacstabwithrelaxedstar}, the stabilized Kac polynomial is exactly
    \[
    \frac{(1-q) p^{|\operatorname{supp} \delta|}(q)}{\prod_{i \notin \operatorname{supp} \delta} \prod_{k=1}^{d_i} (1 - q^k)}.
    \]
\end{conjecture}

As in Remark \ref{r:geometrickacstabinterpr}, note that if $\delta$ satisfies weak (\ref{eq:condition}) for the Cartan form of $Q$, then it still satisfies (\ref{eq:condition}) for the Euler form of $\bar Q$. Then our conjecture is equivalent to the inclusion (\ref{eq:inclofmintorep}) inducing an isomorphism on the first $k$ cohomologies for $n$ large enough.

We also conjecture that Kirwan surjectivity applies to this case as well:
\begin{conjecture}\label{conj:kirwanforzeroframing}
    Let $d$ be an indivisible dimension vector for $Q$ such that $(d,e_i)<0$ for every $i\in Q_0$. Then, for a generic stability parameter, the Kirwan map for $\mc M^\chi_{d}(Q)$ is surjective.
\end{conjecture}
This conjecture would imply that the coefficients of the Kac polynomial for $d$ are bounded by (\ref{eq:convforsymmquiv}).

\begin{example}
    {\rm
    Let $\g$ be a simply-laced Kac-Moody Lie algebra and $d=\sum d_ie_i$ be an imaginary root such that $(d,e_i)<0$ for all simple roots $e_i$. If Conjecture \ref{conj:kirwanforzeroframing} is true, then $\dim \g_d$ is bounded above by the $1-\frac12(d,d)$ coefficient of 
    $$
    \prod_{i}\frac{1}{\prod_{k=1}^{d_i}(1-q^k)}.
    $$
    }
\end{example}

\begin{example}
    {\rm
    As a particular case of the previous example, consider the quiver from Example \ref{e:hyperbolicquiver} and the corresponding Kac-Moody Lie algebra $\g$. Let $d=(n,n,1)$. Then 
    $$
    1-\frac12(d,d)=n
    $$
    and the previous example gives us a bound
    $$
    \dim \g_d\le p^{2}(n).
    $$
    In contrast, Frenkel's conjecture \cite[(4.19)]{F85} asserts that 
$$
\dim \g_d\le p(n),
$$
and it is known to be an equality in this case \cite[1.13]{FF83}.
    }
\end{example}

\begin{remark}
    {\rm
        The appearance of $p^{|\supp\delta|}(q)$ in our formulas suggests that it is reasonable to study implications of our results for Frenkel's conjecture. We point out that there exist other approaches to studying root multiplicities of Kac-Moody Lie algebras using quiver varieties: see \cite{Tin21} and \cite{CT25}.
    }
\end{remark}    

\section{Stabilization for Nakajima quiver varieties}
Let $\delta$ be an imaginary root for $Q$ satisfying (\ref{eq:condition}) for the Cartan form of $Q$. We fix a dimension vector $d$ of $Q$ and a framing $w$. 
\begin{theorem}\label{t:nakajimastab}\
\begin{enumerate}
    \item [(i)] Assume that $\delta$ is an imaginary root for $Q$ satisfying weak {\rm (\ref{eq:condition})} for the Cartan form of $Q$ and $(d,e_i)<0$ for any $i\in Q_0$. Then the cohomologies
    $$
    H^{2k}(\mc M(d+n\delta,w),\C)
    $$
    stabilize for large $n$. Moreover, all these cohomologies are bounded above by the corresponding coefficients of
    $$
    \frac{p^{|\operatorname{supp} \delta|}(q)}{\prod_{i \notin \operatorname{supp} \delta} \prod_{k=1}^{d_i} (1 - q^{k})}.
    $$
    \item[(ii)] Assume that $\supp w\cap\supp \delta\ne \varnothing$ and that $\delta$ is an imaginary root for $Q$ satisfying {\rm (\ref{eq:condition})} for the Cartan form of $Q$. Then 
    $$
    \lim_{k\to\infty} H^{2k}(\mc M(d+n\delta,w),\C)q^k=\frac{p^{|\operatorname{supp} \delta|}(q)}{\prod_{i \notin \operatorname{supp} \delta} \prod_{k=1}^{d_i} (1 - q^{k})}.
    $$
\end{enumerate}
\end{theorem}
\begin{proof}
    Using the Crawley-Boevey trick, we can realize the Najakima quiver variety $\mc M(d+n\delta,w)$ as $\mc M_{(d+n\delta,1)}(Q_w)$. The assumptions made ensure that the necessary assumptions of Theorems \ref{t:kacstab} and Theorem \ref{t:kacstabwithrelaxedstar} are satisfied in the corresponding cases (note that the vector $(d+n\delta,1)$ is always indivisible). Since Kirwan surjectivity is satisfied for Nakajima quiver varieties, the boundedness from (i) follows.
\end{proof}

\begin{example}
    {\rm
    Let $Q$ be a quiver on an affine ADE Dynkin diagram with $r$ nodes, and let $0$ denote a vertex whose deletion gives a connected Dynkin diagram, of the same type. Let $\delta$ be the primitive imaginary root, $d=0$ and $w=(0,\ldots,0,1)$, where $1$ is attached to the vertex $0$. It is easy to see that $\delta$ satisfies the assumptions of Theorem \ref{t:nakajimastab}(i), which gives 
    $$
    \lim_{n\to\infty}\dim H^{2k}(\mc M(n\delta,w),\C)\le p^r(k).
    $$
    
    According to \cite[Theorem 43]{Kuz07}, there exists a generic stability parameter $\chi$ for which
    $$
    \mc M^\chi(n\delta,w)\simeq \operatorname{Hilb}^n(\widehat{\C^2/G}),
    $$
    the Hilbert scheme of $n$ points on $\widehat{\C^2/G}$, the minimal resolution of $\C^2\git G$, where $G$ is the finite subgroup of $\operatorname{SL}_2\C$ corresponding to $Q$ via the MacKay correspondence. Combining \cite[Theorem 6.1]{Nak99} and \cite[Theorem 7.1]{Oda}, we get
    \begin{align}\label{eq:hilbertschemecohomo}
        \sum_{k,n} \dim H^{2k}(\operatorname{Hilb}^n(\widehat{\C^2/G}),\C)t^kq^n&=\prod_{m=1}^\infty\frac{1}{(1-t^{m-1}q^m)(1-t^mq^m)}\\\nonumber
        &=p^{r-2}(q)\prod_{m=1}^\infty\frac{1}{(1-t^{-1}(tq)^m)}.
    \end{align}
    
    Denote by $p_n(m)$ the number of partitions of $m$ in exactly $n$ parts, and by $p_n(q)$ the corresponding generating function. Then
    $$
    \prod_{m=1}^\infty\frac{1}{(1-t^{-1}(tq)^m)}=\sum_{m=0}^\infty t^{-m}p_m(tq)=\sum_{m=0}^\infty t^m \sum_{a=0}^\infty p_a(m+a)q^{m+a}.
    $$
    
    Thus, the $t^kq^{k+a}$ coefficient of \ref{eq:hilbertschemecohomo}  equals
    \begin{align*}
        \sum_{m+n=k}p_a(m+a)p^{r-2}(n)=\sum_{m+n=k}(p_0(m)+p_1(m)+\ldots+p_a(m))p^{r-2}(n),
    \end{align*}
    which is increasing with $a$ and for $a\ge k$ becomes
    $$
\sum_{m+n=k}p(m)p^{r-2}(n)=p^{r-1}(k),
    $$
    supporting Conjecture \ref{conj:exactlimit}.
}
\end{example}

\begin{example}
{\rm
    As a concrete case of the previous example, consider the 2-Kronecker quiver with $d=(0,0)$,  $\delta:=(1,1)$, and framing $(0,1)$:
    
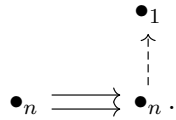
\[\begin{tikzcd}
	& {\bullet_1} \\
	{\bullet_n} & {\bullet_n}
	\arrow[shift left, from=2-1, to=2-2]
	\arrow[shift right, from=2-1, to=2-2]
	\arrow[dashed, from=2-2, to=1-2]\fullstopbelow
\end{tikzcd}\]

In this case, $G=\{\pm I\}\subset \operatorname{SL}_2\C$, which gives
\begin{align*}
    \C^2\git G&\simeq \operatorname{Spec}\frac{\C[x,y,z]}{x^2-yz},\\
    \widehat{\C^2/G}&\simeq T^*\mathbb P^1,
\end{align*}
which is the Springer resolution of the nilpotent cone $\mc N\simeq \C^2\git G$ for $\operatorname{GL}_2$. We have
$$
    \lim_{n\to\infty}\dim H^{2k}(\operatorname{Hilb}^n(T^*\mathbb P^1),\C)=p^2(k).
    $$
}
\end{example}

\section*{Acknowledgments}
The author thanks Samuel DeHority, Igor Frenkel, Lucien Hennecart and Ivan Losev for helpful discussions.


\section*{Declarations}
\noindent \textbf{Conflict of Interest:} The author has no conflicts of interest to declare.

\bibliographystyle{alphaurl}
\nocite{*}
\bibliography{references}

\begin{thebibliography}{CBVdB04}

\bibitem[CB01]{CB01}
William Crawley-Boevey.
\newblock Geometry of the moment map for representations of quivers.
\newblock {\em Compositio Mathematica}, 126(3):257--293, 2001.

\bibitem[CBVdB04]{CBVdB}
William Crawley-Boevey and Michel Van~den Bergh.
\newblock Absolutely indecomposable representations and {K}ac-{M}oody {L}ie algebras.
\newblock {\em Inventiones Mathematicae}, 155(3):537--559, 2004.

\bibitem[CT25]{CT25}
Patrick Chan and Peter Tingley.
\newblock Quiver varieties and root multiplicities in rank 3.
\newblock 2025.
\newblock \href {https://arxiv.org/abs/2508.04804} {\path{arXiv:2508.04804}}.

\bibitem[FF83]{FF83}
Alex~J. Feingold and Igor~B. Frenkel.
\newblock A hyperbolic {K}ac-{M}oody algebra and the theory of {S}iegel modular forms of genus 2.
\newblock {\em Mathematische Annalen}, 263:87--144, 1983.
\newblock URL: \url{https://eudml.org/doc/163734}.

\bibitem[Fre85]{F85}
I.~B. Frenkel.
\newblock Representations of {K}ac-{M}oody algebras and dual resonance models.
\newblock In M.~Flato, P.~Sally, and G.~Zuckerman, editors, {\em Applications of Group Theory in Physics and Mathematical Physics}, volume~21 of {\em Lectures in Applied Mathematics}, pages 325--353. American Mathematical Society, Providence, RI, 1985.
\newblock URL: \url{https://campuspress.yale.edu/frenkel/publications/}.

\bibitem[Hau10]{Hau}
Tam{\'a}s Hausel.
\newblock Kac’s conjecture from {N}akajima quiver varieties.
\newblock {\em Inventiones Mathematicae}, 181(1):21--37, 2010.

\bibitem[Hen20]{Hen}
Lucien Hennecart.
\newblock Asymptotic behavior of {K}ac polynomials.
\newblock {\em Experimental Mathematics}, 32:294 -- 312, 2020.

\bibitem[HLRV13]{HLRV}
Tam{\'a}s Hausel, Emmanuel Letellier, and Fernando Rodriguez-Villegas.
\newblock Positivity for {K}ac polynomials and {DT}-invariants of quivers.
\newblock {\em Annals of Mathematics}, pages 1147--1168, 2013.

\bibitem[Hua00]{Hua}
Jiuzhao Hua.
\newblock Counting representations of quivers over finite fields.
\newblock {\em Journal of Algebra}, 226(2):1011--1033, 2000.

\bibitem[Kac83]{Kac}
Victor~G. Kac.
\newblock Root systems, representations of quivers and invariant theory.
\newblock In Francesco Gherardelli, editor, {\em Invariant Theory}, pages 74--108, Berlin, Heidelberg, 1983. Springer Berlin Heidelberg.

\bibitem[Kac90]{Kac90}
Victor~G. Kac.
\newblock {\em Infinite-Dimensional {L}ie Algebras}, volume~44 of {\em Cambridge Studies in Advanced Mathematics}.
\newblock Cambridge University Press, Cambridge, UK, 3rd edition, 1990.

\bibitem[KJ16]{Kir}
Alexander Kirillov~Jr.
\newblock {\em Quiver representations and quiver varieties}, volume 174.
\newblock American Mathematical Soc., 2016.

\bibitem[Kuz07]{Kuz07}
Alexander Kuznetsov.
\newblock Quiver varieties and {H}ilbert schemes.
\newblock {\em Moscow Mathematical Journal}, 7(4):673--697, 2007.

\bibitem[MN18]{MGN18}
Kevin McGerty and Thomas Nevins.
\newblock Kirwan surjectivity for quiver varieties.
\newblock {\em Inventiones Mathematicae}, 212(1):161--187, 2018.
\newblock \href {https://doi.org/10.1007/s00222-017-0765-x} {\path{doi:10.1007/s00222-017-0765-x}}.

\bibitem[Nak98]{Nak98}
Hiraku Nakajima.
\newblock Quiver varieties and {K}ac-{M}oody algebras.
\newblock {\em Duke Mathematical Journal}, 1998.

\bibitem[Nak99]{Nak99}
Hiraku Nakajima.
\newblock {\em Lectures on {H}ilbert Schemes of Points on Surfaces}, volume~18 of {\em University Lecture Series}.
\newblock American Mathematical Society, Providence, RI, 1999.

\bibitem[Oda21]{Oda}
Tomoki Oda.
\newblock Kleinian singularities.
\newblock Honors thesis, University of California San Diego, Department of Mathematics, 2021.

\bibitem[Rei03]{Rei03}
Markus Reineke.
\newblock The {H}arder-{N}arasimhan system in quantum groups and cohomology of quiver moduli.
\newblock {\em Inventiones Mathematicae}, 152(2):349--368, 2003.

\bibitem[Rei08]{Rei08}
Markus Reineke.
\newblock Moduli of representations of quivers.
\newblock {\em Trends in representation theory of algebras and related topics}, pages 589--638, 2008.

\bibitem[Tin21]{Tin21}
Peter Tingley.
\newblock A quiver variety approach to root multiplicities.
\newblock {\em Algebraic Combinatorics}, 4(1):163--174, 2021.
\newblock \href {https://doi.org/10.5802/alco.158} {\path{doi:10.5802/alco.158}}.

\end{thebibliography}

\end{document}